\documentclass[11pt,leqno]{article}      
\usepackage{graphicx}

\usepackage{latexsym}
\usepackage[cp1250]{inputenc}
\usepackage[OT4]{fontenc}
\usepackage{amsfonts}
\usepackage{amsmath}
\usepackage{amssymb}

\newcommand{\R}{\mathbb{R}}
\newcommand{\Z}{\mathbb{Z}}
\newcommand{\Ker}{\operatorname{Ker}}
\newcommand{\rank}{\operatorname{rank}}
\newcommand{\corank}{\operatorname{corank}}
\newcommand{\codim}{\operatorname{codim}}
\newcommand{\sgn}{\operatorname{sgn}}
\newcommand{\signature}{\operatorname{signature}}
\newcommand{\Aa}{\mathcal{A}}
\newcommand{\inv}{^{-1}}

\newtheorem{theorem}{Theorem}[section]
\newtheorem{lemma}[theorem]{Lemma}
\newtheorem{defin}[theorem]{Definition}
\newtheorem{cor}[theorem]{Corollary}
\newtheorem{prop}[theorem] {Proposition}

\newenvironment{ex}{\medskip \noindent {\bf Example.\ }}{\bigskip}
\newenvironment{proof}{\par\noindent \emph{Proof. }}{\hspace*{\fill}$\Box$\par\medskip}

\title{Cross-cap singularities counted with sign\thanks{%
Iwona~Krzy\.{z}anowska\\
University of Gda\'{n}sk, Institute of Mathematics \\
              80-952 Gda\'{n}sk, Wita Stwosza 57, Poland\\          
              Email: Iwona.Krzyzanowska@mat.ug.edu.pl\\ \\
2000 \emph{Mathematics Subject Classification}:  14P25, 57R45, 57R42, 12Y05\\
\emph{Keywords}: cross-cap, immersion, Stiefel manifold, intersection number, signature\\\\\emph{Supported by National Science Centre, grant 6093/B/H03/2011/40}
}
}

\author{Iwona Krzy\.{z}anowska}

\date{June 2015}

\begin{document}

\def\nothanksmarks{\def\thanks##1{\protect\footnotetext[0]{\kern-\bibindent##1}}}

\nothanksmarks

\maketitle

%
%
%
%
%
%
%
%

\begin{abstract}
There is presented a method for computing the algebraic number of cross-cap singularities for mapping from $m$ dimensional compact manifold with boundary $M\subset \R^m$ into $\R^{2m-1}$, $m$ is odd. As an application, the intersection number of an immersion $g:S^{m-1}(r)\longrightarrow\R^{2m-2}$ is described as the algebraic number of cross-caps of a mapping naturally associated with $g$. 
\end{abstract}


\section{Introduction}


Mappings from the $m$--dimensional, smooth, orientable manifold $M$ into $\R^{2m-1}$ are natural object of study. In \cite{WhitneyUmbrella2},Whitney described typical mappings from $M$ into $\R^{2m-1}$. Those mappings have only isolated critical points, called cross-caps.

According to \cite[Theorem 4.6]{golub},  \cite[Lemma 2]{WhitneyUmbrella}, a mapping $M\to \R^{2m-1}$ has a cross-cap at $p\in M$, if and only if in the local coordinate system near $p$ this mapping has the form $$(x_1,\ldots , x_m)\mapsto(x_1^2,x_2,\ldots , x_m,x_1x_2,\ldots ,x_1x_m).$$

In \cite{WhitneyUmbrella}, for $m$ odd, Whitney presented a method to associate a sign with a cross-cap. Put $\zeta(f)$ -- an algebraic sum of cross-caps of $f:M\longrightarrow\R^{2m-1}$, where $M$ is $m$--dimensional compact orientable manifold. Then according to Whitney, \cite[Theorem 3]{WhitneyUmbrella}, $\zeta(f)=0$, if $M$ is closed. If  $M$ has a boundary, then following Whitney, \cite[Theorem 4]{WhitneyUmbrella}, for a homotopy $f_t:M\rightarrow \R^{2m-1}$ regular in some open neighbourhood of $\partial M$, if the only singular points of $f_0$ and $f_1$ are cross-caps, then $\zeta(f_0)=\zeta(f_1)$. Moreover arbitrary close to any mapping $h:M\rightarrow \R^{2m-1}$, there is mapping regular near boundary, with only cross-caps as singular points (see \cite{WhitneyUmbrella}). In the case where $m$ even, it is impossible to associate sign with cross-cap in the same way as in the odd case, but if $m$ is even, it is enough to consider number of cross-caps $\operatorname{mod} 2$, to get similar results (see \cite{WhitneyUmbrella}).

In \cite{KrzyzSzafr}, the authors studied a mapping $\alpha$ from a compact and oriented $(n-k)$--manifold $M$ into the Stiefel manifold $\widetilde{V}_k(\R^n)$, for $n-k$ even. They constructed a mapping $\widetilde{\alpha}:S^{k}\times M\rightarrow \R^n\setminus\{0\}$ associated with $\alpha$, and defined $\Lambda(\alpha)$ as half of topological degree of $\widetilde{\alpha}$. In case $M=S^{n-k}$, they showed that $\Lambda(\alpha)$ corresponds with the class of $\alpha$ in $\pi_{n-k}\widetilde{V}_k(\R^n)\simeq\Z$. 

According to  \cite{KrzyzSzafr}, in the case where $M\subset\R^{n-k+1}$ is an algebraic hypersurface, $\alpha$ is polynomial, with some more assumptions $\Lambda(\alpha)$ can be presented as a sum of signatures of two quadratic forms defined on $\R[x_1,\ldots ,x_{n-k+1}]$. And so, easily computed.

In this paper we prove that in the case where $m$ is odd,  for $f:(M,\partial M)\longrightarrow \R^{2m-1}$, where $M\subset \R^m$, $\zeta(f)$ can be expressed  as $\Lambda(\alpha)$, for some $\alpha$ associated with $f$. And so, with some more assumptions, $\zeta(f)$  can be easily computed for polynomial mapping $f$. Moreover we present a method that can be used  to check effectively that $f$ has only cross-caps as singular points.

Take a smooth map $g:\R^m\rightarrow\R^{2m}$, let us assume that $g|_{S^{m-1}}$ is an immersion. In this paper we show that the intersection number $I(g|_{S^{m-1}})$ of immersion $g|_{S^{m-1}}$,  can be presented as an algebraic sum of cross-caps of the mapping $(\omega,g)|{B^m}$, where $\omega$ is sum of squares of coordinates.

Take $f:(\R^m,0)\longrightarrow\R^{2m-1}$ with cross-cap at $0$. In \cite{IkegamiSaeki}, Ikegami and Saeki defined the sign of a cross-cap singularity for mapping $f$ as the intersection number of immersion $f|_S:S=f\inv(S^{2m-2}(\epsilon))\longrightarrow S^{2m-2}(\epsilon)$, for $\epsilon$ small enough. It is easy to see that this definition complies with Whitney definition from \cite{WhitneyUmbrella}. In \cite{IkegamiSaeki}, the authors showed that for generic map (in sense of \cite{IkegamiSaeki}) $g:(\R^m,0)\longrightarrow\R^{2m-1}$, the number of Whitney umbrellas appearing in a $C^{\infty}$ stable perturbation of $g$, counted with signs, is an invariant of the topological $\mathcal{A}_+$-- equivalence
class of $g$, and is equal to the intersection number of $g|_S:S=g\inv(S^{2m-2}(\epsilon))\longrightarrow S^{2m-2}(\epsilon)$.
Using our methods this number can be easily computed for polynomial mappings.

We use notation $S^n(r)$, $B^n(r)$, $\bar{B}^n(r)$ for sphere, open ball, closed ball (resp.) centred at the origin of radios $r$ and dimension $n$. If we omit symbol $r$, we assume that $r=1$.

\section{Cross-cap singularities}\label{sekcjaKiedy}

Let $M,N$ be smooth manifolds. Take a smooth mapping $f:M\longrightarrow N$.

\begin{lemma}\label{oTranswersalnosci}
 
Let $W$ be a submanifold of $N$. Take $p\in M$ such that $f(p)\in W$. Let us assume that there is a neighbourhood  $U$ of $f(p)$ in $N$ and a smooth mapping $\phi\colon U \to \R ^s$ such that $\rank D\phi(f(p)) = k=\codim W$ and $W\cap U=\phi \inv (0)$. Then $f\pitchfork W$ at  $p$  if and only if $\rank D(\phi\circ f)(p)=k$.     
\end{lemma}

\begin{proof} Of course $\Ker D\phi(f(p))=T_{f(p)}W$, so $\dim T_{f(p)}N=\dim \Ker D\phi(f(p))+k$. Then:
\[ f\pitchfork W \mbox{ at } p\ \Leftrightarrow \ T_{f(p)}N=T_{f(p)}W+Df(p)T_pM\ \Leftrightarrow \]   \[ \Leftrightarrow \ T_{f(p)}N=\Ker D\phi(f(p))+Df(p)T_pM.\]
This equality holds if and only if there exist vectors $v_1,\ldots , v_k\in Df(p)T_pM$, such that any nontrivial combination of $v_1,\ldots ,v_k$ is outside the $\Ker D\phi(f(p))$ and so $\rank D\phi(f(p))\left[v_1\ldots v_k\right] =k$. We get that $f\pitchfork W$ at  $p$ if and only if $\rank D(\phi \circ f)(p)=k .$
\end{proof}

By $j^1f$ we mean the canonical mapping associated with $f$, from $M$ into the spaces of $1$--jets $J^1(M,N)$.
We say that $f$ is $1$--generic, if $j^1f\pitchfork S_r$, for $r\geq 0$, where $S_r=\{\sigma\in J^1(M,N)\ |\ \corank \sigma=r\}$.

Let us assume that $M$ and $N$ are manifolds of dimension $m$ and $2m-1$ respectively. In this case (see \cite{golub}) $\codim S_r=r^2+r(m-1)$, and so $\codim S_1=m$ and $\codim S_r>m$, for $r\geq 2$. So $f$ is $1$--generic if and only if $f\pitchfork S_1$ and $S_r(f)=\emptyset$ for $r\geq 2$.

\begin{defin}
We say that $p\in S_1(f)$ is a cross-cap singularity of  $f$, if $j^1f\pitchfork S_1$ at $p$. 
\end{defin}

According to  \cite{WhitneyUmbrella2,WhitneyUmbrella}, a point $p$ is a cross-cap of $f:M\longrightarrow \R^{2m-1}$, if there is a coordinate system near $p$, such that \begin{equation} \label{warunek1}
\frac{\partial f}{\partial x_1}(p)=0
\end{equation} and vectors \begin{equation}\label{wektory}
\frac{\partial^2 f}{\partial x_1^2}(p),\frac{\partial f}{\partial x_2}(p),\ldots ,\frac{\partial f}{\partial x_m}(p),\frac{\partial^2 f}{\partial x_1\partial x_2}(p),\ldots , \frac{\partial^2 f}{\partial x_1\partial x_m}(p)
\end{equation} are linearly independent. According to  \cite[Section 2]{WhitneyUmbrella2}, if $p$ is a cross-cap singularity and (\ref{warunek1}) holds, then vectors (\ref{wektory}) are linearly independent.\\
This two above definitions of a cross-cap singularity are equivalent (see \cite{golub}, \cite{WhitneyUmbrella}). 
%

Take $f:\R^m\longrightarrow \R^{2m-1}$. Put $\mu:\R^m\longrightarrow\R^s$ such that  $\mu(x)$ is given by all the $m$--minors of $Df(x)$. Of course $s=\binom{2m-1}{m}$. 
\begin{lemma}\label{kiedyCrossCup}
A point $p\in \R^m$ is a cross-cap singularity of $f$ if and only if $\rank Df(p)=m-1$ and $\rank D\mu(p)=m$.
\end{lemma}

\begin{proof} A point $p$ is a cross-cap singularity if and only if $p\in S_1(f)$ and $j^1f\pitchfork S_1$ at $p$. Note that $p\in S_1(f)$  if and only if $\rank Df(p)=m-1$.

Of course $J^1(\R^m,\R^{2m-1})\cong\R^m\times\R^{2m-1}\times M(2m-1,m)$, where $M(2m-1,m)$ is a space of real matrices of dimension  $(2m-1)\times m$.  Take $U$ -- an open neighbourhood of $j^1f(p)$ in $J^1(\R^m,\R^{2m-1})$, and a mapping
$$\phi :U\longrightarrow \R^s,$$ where $\phi(x,y,[a_{ij}])$ is given by all $m$--minors of $[a_{ij}]$.
We may assume that $$\det\dfrac{\partial (f_1,\ldots , f_{m-1})}{\partial (x_1,\ldots ,x_{m-1})}(p)\neq 0.$$
Put $A=[a_{ij}]$, for $1\leq i,j\leq m-1$, then for $U$ small enough, $\det A\neq 0$. Let $M_i$ be the determinant of submatrix of  $[a_{ij}]$ composed of first $m-1$ rows and row number $(m+i-1)$, for $i=1,\ldots ,m$. Then $$M_i=(-1)^{2m-1+i}\det A \cdot a_{m+i-1,m}+ b_i,$$ for $i=1,\ldots,m$ and $b_i$ does not depend on $a_{mm},\ldots ,a_{2m-1,m}$, and so $$\rank\frac{\partial(M_1,\ldots , M_m)}{\partial(a_{m,m},\ldots , a_{2m-1,m})}=m.$$ 

We get that $$\rank D\phi(j^1f(p))\geq m.$$ Let us recall that $\codim S_1=m$. We can choose $U$ small enough such that  $$\phi \inv (0)=U\cap S_1.$$ So we get that $\rank D\phi(j^1f(p))=\codim S_1 =m$. Of course $\phi\circ j^1f =\mu$ in the small neighbourhood of $p$. According to Lemma \ref{oTranswersalnosci}, $j^1f\pitchfork S_1$ at $p$ if and only if $\rank D\mu(p)=m$.

\end{proof}

\section{Algebraic sum of cross-cap singularities}\label{sekcjaSuma}

Let $m$ be odd. Take a smooth mapping $f:M\longrightarrow\R^{2m-1}$ and  let $p\in M$ be a cross-cap of $f$. According to \cite{WhitneyUmbrella}, $p$ is called positive (negative) if  the vectors (\ref{wektory}) determine the negative  (positive) orientation of $\R^{2m-1}$. According to \cite[Lemma 3]{WhitneyUmbrella}, this definition does not depend on choosing the coordinate system on $M$.

Let us assume, that $f:\R^m\longrightarrow\R^{2m-1}$ is a smooth mapping such that $0$ is a cross-cap of $f$. Of course it is an isolated critical point of $f$. Denote by $v_i$ the i--th column of $Df$, for $i=1,\ldots ,m$. There exists $r>0$ such that $v_1(x),\ldots ,v_m(x)$ are independent on $\bar{B}^m(r)\setminus \{0\}$. Following \cite{KrzyzSzafr} we can define $$\widetilde{\alpha}(\beta,x)=(\beta_1v_1(x)+\ldots +\beta_mv_m(x))=Df(x)(\beta):S^{m-1}\times \bar{B}^m(r)\longrightarrow \R^{2m-1}.$$ Then the topological degree of the mapping $$\widetilde{\alpha}|_{S^{m-1}\times S^{m-1}(r)}:S^{m-1}\times S^{m-1}(r)\longrightarrow\R^{2m-1}\setminus \{0\}$$is well defined. By \cite[Proposition 2.4]{KrzyzSzafr}, $\deg(\widetilde{\alpha}|_{S^{m-1}\times S^{m-1}(r)})$ is even. 

\begin{theorem}\label{oParasolu} Let $m$ be odd. If $0$ is a cross-cap of a mapping $f:\R^m\longrightarrow\R^{2m-1}$, then it is positive if and only if $\frac{1}{2}\deg(\widetilde{\alpha}|_{S^{m-1}\times S^{m-1}(r)})=-1$, and so it is negative if and only if $\frac{1}{2}\deg(\widetilde{\alpha}|_{S^{m-1}\times S^{m-1}(r)})=+1$.
\end{theorem}
\begin{proof} We can find linear coordinate system $\phi:\R^m\longrightarrow\R^m$, such that $\phi(0)=0$ and $f\circ \phi$ fulfills condition (\ref{warunek1}) at $0$.  Denote by $A$ the matrix of $\phi$. Let $w_1,\ldots , w_m$ denote columns of $D(f\circ \phi)$. Then $w_1(0)=0$ and vectors \begin{equation}\label{wektory2}
\frac{\partial w_1}{\partial x_1}(0),w_2(0),\ldots ,w_m(0),\frac{\partial w_1}{\partial x_2}(0),\ldots , \frac{\partial w_1}{\partial x_m}(0) 
\end{equation} are linearly independent (see \cite[Section 2]{WhitneyUmbrella2}).
Put $\widetilde{\gamma}(\beta,x)=(\beta_1w_1(x)+\ldots +\beta_mw_m(x)):S^{m-1}\times \bar{B}^{m}(r)\longrightarrow\R^{2m-1}$.  We can assume that $r$ is such that $\widetilde{\gamma}\neq 0$ on $S^{m-1}\times \bar{B}^{m}(r)\setminus\{0\}$.  Let us see that $$\widetilde{\gamma}(\beta ,x)=D(f\circ \phi)(x)\cdot \left [
\begin{array}{c}
\beta_1\\
\vdots\\
\beta_m
\end{array} \right ]=Df(\phi(x))\cdot A\cdot \left [
\begin{array}{c}
\beta_1\\
\vdots\\
\beta_m
\end{array} \right ]=Df(\phi(x))\cdot \left [
\begin{array}{c}
\phi_1(\beta)\\
\vdots\\
\phi_m(\beta)
\end{array} \right ].$$ So $\widetilde{\gamma}=\widetilde{\alpha}\circ(\phi\times \phi)$. It is easy to see that $\phi\times\phi$ preserve the orientation of $S^{m-1}\times S^{m-1}(r)$. We can assume that $r>0$ is so small, that $\deg(\widetilde{\alpha}|_{S^{m-1}\times S^{m-1}(r)})=\deg(\widetilde{\alpha}|_{\phi(S^{m-1})\times\phi( S^{m-1}(r))})$. So we get that $$\deg(\widetilde{\gamma}|_{S^{m-1}\times S^{m-1}(r)})=\deg(\widetilde{\alpha}|_{\phi(S^{m-1})\times\phi( S^{m-1}(r))})\deg(\phi\times \phi)=\deg(\widetilde{\alpha}|_{S^{m-1}\times S^{m-1}(r)}).$$

On the other hand $\deg(\widetilde{\gamma}|_{S^{m-1}\times S^{m-1}(r)})=\deg(\widetilde{\gamma},0,S^{m-1}\times \bar{B}^m(r))$. Since $f\circ \phi$ fulfils  (\ref{warunek1}), vectors $w_2,\ldots ,w_m$ are independent on $\bar{B}^m(r)$. Let us see that $\widetilde{\gamma}=0$ on $S^{m-1}\times \bar{B}^m(r)$ if and only if $x=0$ and $\beta=(\pm 1,0,\ldots ,0)$. So $\deg(\widetilde{\gamma}|_{S^{m-1}\times S^{m-1}(r)})$ is a sum of the local topological degree of $\widetilde{\gamma}$ at $(1,0,\ldots ,0;0,\ldots ,0)$ and at $(-1,0,\ldots ,0;0,\ldots ,0)$. 

Near $(1,0,\ldots ,0;0,\ldots ,0)$ the well--oriented parametrisation of $S^{m-1}\times \bar{B}^m(r)$ is given by $$(\beta_2,\ldots , \beta_m;x)=(\sqrt{1-\beta_2^2-\ldots -\beta_m^2},\beta_2,\ldots , \beta_m;x).$$ And then the derivative matrix of $\widetilde{\gamma}$ at $(1,0,\ldots ,0;0,\ldots ,0)$ has a form $$ A_1=\left [
\begin{array}{cccccc}
w_2(0)&\ldots &w_m(0)&\frac{\partial w_1}{\partial x_1}(0)&\ldots &\frac{\partial w_1}{\partial x_m}(0)
\end{array} \right ].$$
Near $(-1,0,\ldots ,0;0,\ldots ,0)$ the well--oriented parametrisation  of $S^{m-1}\times \bar{B}^m(r)$ is given by $$(\beta_2,\ldots , \beta_m;x)=(-\sqrt{1-\beta_2^2-\ldots -\beta_m^2},-\beta_2,\ldots , \beta_m;x).$$ And then the derivative matrix of $\widetilde{\gamma}$ at $(-1,0,\ldots ,0;0,\ldots ,0)$ has a form $$ A_2=\left [
\begin{array}{cccccc}
-w_2(0)&\ldots &w_m(0)&-\frac{\partial w_1}{\partial x_1}(0)&\ldots &-\frac{\partial w_1}{\partial x_m}(0)
\end{array} \right ].$$

Let us recall that $m$ is odd. System of vectors (\ref{wektory2}) is independent, so $0$ is a regular value of $\widetilde{\gamma}$, and $$\frac{1}{2}\deg(\widetilde{\gamma}|_{S^{m-1}\times S^{m-1}(r)})=\frac{1}{2}(\sgn\det A_1+\sgn\det A_2)=\sgn\det A_1.$$ Moreover $0$ is a positive cross-cap if and only if vectors (\ref{wektory2}) determine negative orientation of a $\R^{2m-1}$, i.e  if and only if $\frac{1}{2}\deg(\widetilde{\alpha}|_{S^{m-1}\times S^{m-1}(r)})=-1$.
\end{proof}

Let $U\subset\R^{m}$ be an open bounded set and $f:\overline{U}\longrightarrow \R^{2m-1}$ be smooth. We say that $f$ is generic if only critical points of $f$ are cross-caps and $f$ is regular in the neighborhood of $\partial U$. Let us  denote by $\zeta(f)$ the algebraic sum of cross-caps of $f$.

\begin{prop}\label{oParasolach}
Let $U\subset \R^m$, ($m$ is odd), be a bounded $m$--dimensional manifold such that $\overline{U}$ is an $m$--dimensional manifold with a boundary. For $f:\overline{U}\subset \R^m\longrightarrow\R^{2m-1}$ generic, $\zeta(f)=-\frac{1}{2}\deg(\widetilde{\alpha})$, 
where $\widetilde{\alpha}(\beta,x)=Df(x)(\beta):S^{m-1}\times \partial U \longrightarrow \R^{2m-1}\setminus\{0\}$.

\end{prop}
\begin{prop}\label{oParasolachGeneralnie}
Let $U\subset \R^m$, ($m$ is odd), be a bounded $m$--dimensional manifold  such that $\overline{U}$ is an $m$--dimensional manifold with a boundary. Take $h:\overline{U}\subset \R^m\longrightarrow\R^{2m-1}$  a smooth mapping  such that $h$ is regular in a neighborhood of $\partial U$. Then for every generic $f:\overline{U}\subset \R^m\longrightarrow\R^{2m-1}$ close enough to $h$ in $C^1$--topology we have, $\zeta(f)=-\frac{1}{2}\deg(\widetilde{\alpha})$, where $\widetilde{\alpha}(\beta,x)=Dh(x)(\beta):S^{m-1}\times \partial U \longrightarrow \R^{2m-1}\setminus\{0\}$.
\end{prop}

\section{Examples}\label{ex}

To compute some examples we want first to recall the theory presented in  \cite{KrzyzSzafr}. 

Take $\alpha=(\alpha_1,\ldots ,\alpha_k):\R^{n-k+1}\longrightarrow M_k(\R^n)$ 
 a polynomial mapping, $n-k$ even. By $\left [a_{ij}(x)\right ]$, 
$1\leq i\leq n$, $1\leq j\leq k$, we denote the matrix given by $\alpha(x)$ (i.e. $\alpha_j(x)$ 
stands in the $j$--th column). Then 
$$\widetilde{\alpha}(\beta, x)=
\beta_1\alpha_1(x)+\ldots +\beta_k\alpha_k(x)=\left [a_{ij}(x)\right ]\left [\begin{array}{c}
\beta_1\\
\vdots\\
\beta_k
\end{array} \right ]:\R^k\times \R^{n-k+1}\longrightarrow \R^n.$$
Let $I$ be the ideal in $\R[x_1,\ldots ,x_{n-k+1}]$ generated by all 
$k\times k$ minors of  $\left [a_{ij}(x)\right ]$, and $V(I)=\{x\in\R^{n-k+1}\ |\  h(x)=0\mbox{ for all }h\in I\}$.

Take 
$$m(x)=\det\left [\begin{array}{ccc}
a_{12}(x)&\ldots &a_{1k}(x)\\ \\
a_{k-1,2}(x)&\ldots &a_{k-1,k}(x) 
\end{array} \right ].$$

For $k\leq i\leq n$, we define
$$\Delta_i(x)=\det\left[\begin{array}{ccc}
a_{11}(x)&\ldots &a_{1k}(x) \\
 &\ldots& \\
a_{k-1,1}(x)&\ldots &a_{k-1,k}(x)\\
a_{i1}(x)&\ldots &a_{ik}(x)
\end{array} \right ]\ .$$

Put $\Aa=\R[x_1,\ldots,x_{n-k+1}]/I$. Let us assume that $\dim\Aa<\infty$, so that $V(I)$ is finite.
For $h\in\Aa$, we denote by $T(h)$ the trace of the linear endomorphism
$\Aa\ni a\mapsto h\cdot a\in\Aa$. Then $T:\Aa\rightarrow\R$ is a linear functional.

Let $u\in\R[x_1,\ldots,x_{n-k+1}]$. Assume that $\bar{U}=\{x\ |\ u(x)\geq 0\}$ is bounded and
$\nabla u(x)\neq 0$ at each $x\in u\inv(0)=\partial U$. Then $\bar{U}$ is a compact manifold with boundary, and $\dim \bar{U}=n-k$. 

Put $\delta=\partial(\Delta_k,\ldots,\Delta_n)/\partial(x_1,\ldots,x_{n-k+1})$. With $u$ and $\delta$ we associate
quadratic forms $\Theta_\delta,\ \Theta_{u\cdot \delta}:\Aa\rightarrow\R$
given by $\Theta_\delta(a)=T(\delta\cdot a^2)$ and $\Theta_{u\cdot \delta}(a)=T(u\cdot\delta\cdot a^2)$.

\begin{theorem}\cite[Theorem 3.3]{KrzyzSzafr} If $n-k$ is even, 
$\alpha=(\alpha_1,\ldots ,\alpha_k):\R^{n-k+1}\longrightarrow M_k(\R^n)$ 
is a polynomial mapping such that $\dim\Aa<\infty$, 
$I+\langle m \rangle=\R[x_1,\ldots ,x_{n-k+1}]$ and quadratic forms 
$\Theta_\delta,\, \Theta_{u\cdot\delta}:\Aa\longrightarrow\R$ are non--degenerate, then 
the restricted mapping $\alpha|_{\partial U} $ goes into $\widetilde V_k(\R^n)$ and
$$\Lambda(\alpha|_{\partial U})=\frac{1}{2}\deg(\widetilde{\alpha}|_{S^{k-1}\times \partial U})=
\frac{1}{2}(\signature\Theta_{\delta}+\signature\Theta_{u\cdot\delta}),$$ 
where $\widetilde{\alpha}(\beta,x)=\beta_1\alpha_1(x)+\ldots +\beta_k\alpha_k(x)$.
\end{theorem} 

Using the theory presented in \cite{KrzyzSzafr}, particularly \cite[Theorem 3.3]{KrzyzSzafr}, and computer system {\sc Singular} (\cite{GPS06}), one can apply the results from Sections \ref{sekcjaKiedy} and \ref{sekcjaSuma} 
to compute algebraic sum of cross-caps for polynomial mappings.

\begin{ex}
Let us take $f:\R^3\longrightarrow\R^5$ given by $$f(x,y,z)=(12y^2+z,6x^2+y^2+6y,18xy+13y^2+9x,$$$$8x^2z+10xz^2+5x^2+3xz,x^2y+4xyz+yz+4z^2).$$ 

Applying  Lemma \ref{kiedyCrossCup} and using  {\sc Singular} one can check that $f$ is $1$--generic. Moreover, according to Proposition \ref{oParasolach} and \cite{KrzyzSzafr}, one can check  that $$\zeta(f|_{\bar{B}^3(\sqrt{3})})=2,\quad   \zeta(f|_{\bar{B}^3(10)})=1.$$ We can also check that $f$ has $11$ cross-caps in $\R^3$, $6$ of them are positive, $5$ negative.
\end{ex}

\begin{ex}
Take $f:\R^5\longrightarrow\R^9$ given by $$f(s,t,x,y,z)=(y,z,t,20x^2+17sz+x,13sy+13sz+5t,25st+4x^2+28z,$$$$3x^2+19yz+22s,11ts^2+8t^2z+xz,27txy+9sxz+20st).$$ One may check that $f$ is $1$--generic, has $3$ cross-caps in $\R^5$ and $$\zeta(f|_{\bar{B}^3(1/10)})=0,\quad   \zeta(f|_{\bar{B}^3(2)})=-1,\quad   \zeta(f|_{\bar{B}^3(1000)})=1.$$ 
\end{ex}

\section{Intersection number of immersions}

As in previous Sections we assume that $m$ is odd. Take $g=(g_1,\ldots, g_{2m-2}):\R^{m}\longrightarrow\R^{2m-2}$ a smooth map. Denote by $\omega=x_1^2+\ldots +x_{m}^2$. Then $S^{m-1}(r)=\{x|\ \omega(x)=r^2\}$. According to \cite[Lemma 18]{KarNowSzafr}, $g|_{S^{m-1}(r)}$ is an immersion if and only if $$\rank \left[\begin{array}{ccc}
2x_1&\ldots &2x_{m}\\
\frac{\partial g_1}{\partial x_1}(x)&\ldots &\frac{\partial g_1}{\partial x_{m}}(x)\\
&\ldots&\\
\frac{\partial g_{2m}}{\partial x_1}(x)&\ldots &\frac{\partial g_{2m}}{\partial x_{m}}(x)
\end{array}\right]=m,$$  for $x\in S^{m-1}(r)$.

Take $0<r_1<r_2$, such that $g|_{S^{m-1}(r_1)}$ and $g|_{S^{m-1}(r_2)}$ are immersions. Denote by $P=\{x|\ r_1^2\leq w(x)\leq r_2^2\}$. Then $P$ is an $m$--dimensional oriented manifold with boundary. Then $(\omega,g):\R^m\longrightarrow\R^{2m-1}$ is a regular map in the neighborhood of $\partial P$.  Let us define $\widetilde{\alpha}:S^{m-1}\times P\longrightarrow\R^{2m-1}$ as $$\widetilde{\alpha}(\beta,x)=\left[\begin{array}{ccc}
2x_1&\ldots &2x_{m}\\
\frac{\partial g_1}{\partial x_1}(x)&\ldots &\frac{\partial g_1}{\partial x_{m}}(x)\\
&\ldots&\\
\frac{\partial g_{2m}}{\partial x_1}(x)&\ldots &\frac{\partial g_{2m}}{\partial x_{m}}(x)
\end{array}\right]\left[\begin{array}{c}
\beta_1\\\vdots\\\beta_{m}
\end{array}\right]$$ 

\begin{prop} Let us assume that $g|_{S^{m-1}(r_1)}$ and $g|_{S^{m-1}(r_2)}$ are immersions, then
$$I(g|_{S^{m-1}(r_2)})-I(g|_{S^{m-1}(r_1)})=\zeta((\omega,g)|_P).$$
\end{prop} 

\begin{proof} Let us recall that  $m$ is odd. Then  $$\deg(\widetilde{\alpha},S^{m-1}\times P,0)=\deg(\widetilde{\alpha}|_{{S^{m-1}\times \partial P}})=$$$$=\deg(\widetilde{\alpha}|_{{S^{m-1}\times S^{m-1}(r_2)}})-\deg(\widetilde{\alpha}|_{{S^{m-1}\times S^{m-1}(r_1)}}).$$ According to \cite[Theorem 4.2]{KrzyzSzafr}, we get that  $$\deg(\widetilde{\alpha}|_{{S^{m-1}\times S^{m-1}(r_i)}})=I(g|_{S^{m-1}(r_i)}),$$ for $i=1,2$. Applying Proposition \ref{oParasolachGeneralnie} we get $\zeta((\omega,g)|_P)=-\frac{1}{2}\deg(\widetilde{\alpha}|_{S^{m-1}\times \partial P})$. And so  $$\zeta((\omega,g)|_P)=I(g|_{S^{m-1}(r_2)})-I(g|_{S^{m-1}(r_1)}).$$

\end{proof}

\begin{cor}
If $g|_{S^{m-1}(r)}$ is an immersion, then $$I(g|_{S^{m-1}(r)})=\zeta((\omega,g)|_{\bar{B}^m(r)}).$$
\end{cor}

Note that if the only singular points of $(\omega,g)|_{\bar{B}^m(r)}$ are cross-caps, then the intersection number of an immersion $g|_{S^{m-1}(r)}$ is equal to the algebraic sum of cross-caps of $(\omega,g)|_{\bar{B}^m(r)}$. Also, in generic case,the  difference between intersection numbers of immersions  $g|_{S^{m-1}(r_1)}$ and $g|_{S^{m-1}(r_2)}$, is equal to the algebraic sum of cross-caps of $(\omega,g)$ appearing in $P$.

Moreover, if $(\omega,g)$ has finite number of singular points, and all of them are cross-caps, then for any $R>0$ big enough, $g|_{S^{m-1}(R)}$ is an immersion with the same intersection number equal to the algebraic sum of cross-caps of  $(\omega,g)$.

\begin{ex}
Take $g:\R^3\longrightarrow\R^4$ given by $$g=(-3y^2+5yz-x+2,-4x^2+z^2+9y-6z+5,$$$$4x^2z-2x^2+2xy-y-3,3y^2z+xy-4yz+4x-5y-5),$$ and $\omega=x^2+y^2+z^2$. In the same way as in Section \ref{ex}, one may check that the only singular points of $(\omega,g)$ are cross-caps, moreover $(\omega,g)$ has $8$ cross-caps, $5$ of them are positive and $3$ negative. According to previous results $g|_{S^2(r)}$ is an immersion for all $r>0$, except at most $8$ values of $r$ and if  $g|_{S^2(r)}$ is an immersion, then $$-3\leq I(g|_{S^2(r)})\leq 5$$ and for  $R>0$ big enough $g|_{S^2(R)}$ is an immersion with $$I(g|_{S^2(R)})=2.$$

\end{ex}





\end{document}